\documentclass{article}
\usepackage{graphicx}
\usepackage{epsfig}
\usepackage{pstricks}
\usepackage{pst-node}
\usepackage{amsfonts, amsthm,amsmath}
\newtheorem{definition}{Definition}
\newtheorem{theorem}{Theorem}[section]
\newtheorem{conjecture}{Conjecture}[section]
\newtheorem{lemma}[theorem]{Lemma}
\newtheorem{observation}{Observation}[section]
\newtheorem{question}{Question}

\title{A new approach to b-coloring of regular graphs}

\begin{document}

\author{Magda Dettlaff$^1$\and Hanna Furma\'nczyk $^1$  \and  Iztok Peterin$^{2,3}$ \and Riana Roux$^4$ \and Rados\l{}aw Ziemann$^1$}
\date{}
\maketitle

\noindent {\small $^1$ University of Gdańsk, Institute of Informatics, Faculty of Mathematics, 
Physics and Informatics, Wit Stwosz 57, 80-309 Gdańsk, Poland.   \\
\small	$^2$ University of Maribor, Faculty of Electrical Engineering and Computer Science, Koro\v{s}ka 46, 2000 Maribor, Slovenia.\\ 
\small	$^3$ Institute of Mathematics, Physics and Mechanics, Jadranska 19, 1000 Ljubljana, Slovenia.\\
\small $^4$ Department of Mathematical Sciences, Stellenbosch University, Stellenbosch, South Africa\\
\small E-mails: magda.dettlaff@ug.edu.pl, hanna.furmanczyk@ug.edu.pl, iztok.peterin@um.si, rianaroux@sun.ac.za, radoslaw.ziemann@ug.edu.pl}

\begin{abstract}
 Let $G$ be a graph and $c:V(G)\rightarrow \{1,\dots,k\}$ a proper $k$-coloring of $G$, i.e. $c(u)\neq c(v)$ for every edge $uv$ from $G$. A proper $k$-coloring is a b-coloring if there exists a vertex in every color class that contains all the colors in its closed neighborhood. The maximum number of colors $k$ admitting b-coloring of $G$ is the b-chromatic number $\chi_b(G)$. 
 
 We present two separate approaches to the conjecture posed by Blidia et. al [Discrete Appl. Math., 157(8):1787–1793, 2009] that $\chi_b(G)=d+1$ for every $d$-regular graph of girth at least five except the Petersen graph. 
\end{abstract}

\section{Introduction}

A coloring of a graph $G$ can (sometimes) be locally adjusted to a different coloring by changing the color of a vertex $v$ to any color not in its closed neighborhood. This is possible only when there exists a color that is not present in the closed neighborhood of $v$. If this is possible for every vertex of a fixed color, then a given coloring can be transformed into a (better) coloring with less colors. By repeating this procedure, we eventually end with a coloring where the mentioned recoloring is not possible anymore. This yields the existence of a vertex in every color that cannot be recolored. Such a vertex, called a b-\emph{vertex}, must contain all the colors of a given coloring in its closed neighborhood. Such a procedure is a heuristic approach to determining the chromatic number $\chi(G)$ and the number of colors, when the procedure stops, is an upper bound for $\chi(G)$.

As usually with heuristics, one would like to have an estimation of how much one can differ from $\chi(G)$; the worst case, called the b-chromatic number $\chi_b(G)$, was introduced by Irving and Manlove in 1999 (see \cite{IrMa}). They proved, among other results, that determining $\chi_b(G)$ is NP-hard in general, but solvable in polynomial time for trees. Since then, the topic gain quite some attention in the community and went into different directions. There were several studies that compare $\chi_b(G)$ with another heuristics' worst case scenario the Grundy number $\Gamma(G)$ by Masih and Zacker (see \cite{MaZa} and the references therein). Another related topic is \emph{b-continuous graphs}, where the recoloring heuristics can stop at any number of colors between $\chi(G)$ and $\chi_b(G)$, see Ibiapina and Silva \cite{IbSi} for the latest publication on this topic.  Recently, the \emph{acyclic} b-chromatic number was introduced by Anholcer et al. \cite{AnCP}. For more about the history, the survey \cite{JaPe} is recommended. 

For an application of the model of b-coloring consider the following example. Assume that we want to organize an engaging training session in company $X$. We want to divide people into groups where no two persons know each other or work in the same department. But we also want to ensure that in each such a group there is at least one person who knows at least one person in each other group. 
Such a situation will help to ensure a good information exchange between groups and, as a result, better integration of all teams. Because we do not want individual groups to be too large, we want to divide employees into the maximum number of groups that provide the above-described property. Note that such a problem can be modeled by an appropriate graph $G$ and its b-coloring with $\chi_b(G)$ colors. Vertices of the graph $G$ correspond to workers and two vertices are adjacent if and only if the corresponding employees work in the same department or know each other. The condition that there is a person in each group who will ensure good communication with the other teams is implemented by b-vertices in each of the color classes in a b-coloring of such a graph. We are interested in maximizing the number of groups for training, so we are interested in a b-coloring with $\chi_b(G)$ colors of the obtained graph.

One of the more investigated classes of graphs with respect to the b-chromatic number are regular graphs. The reason for this is probably that every vertex can be a b-vertex in a coloring of a regular graph. In general, every b-vertex must be of big enough degree, at least the number of colors minus one. But in regular graphs all vertices have the same degree and are therefore all candidates for b-vertices in a coloring. The main question for a $d$-regular graph is if $\chi_b(G)=d+1$. This happens quite often and it was shown by Jakovac and Klav\v zar \cite{JaKl} that there are only four exception among all cubic graphs. It was also shown that $d$-regular graphs with large enough girth has $\chi_b(G)=d+1$. The main conjecture for $d$-regular graphs was posed by Blidia et al. \cite{BLIDIA20091787}.  

\begin{conjecture}[\cite{BLIDIA20091787}]
\label{conj}
Every $d$-regular graph of girth at least 5, different than Petersen graph, has a b-coloring with $d+1$ colors.
\end{conjecture}

This conjecture was considered by many authors, see \cite{BLIDIA20091787,CaJa,K2004,SK2009,SKKM2014,SKM2015}, where it was shown that there exists only a finite number of graphs for which the conjecture could be false as explained in the detail in the following section.

In this paper we continue attacking the above mentioned conjecture by two different approaches in Sections \ref{noC6} and \ref{twoBunchd}, respectively. First we deal with the case when there exists a vertex $x$ that does not belong to too many six cycles which have all vertices on a distance at most two from $x$. In particular, we prove Conjecture \ref{conj} when there is such a vertex with at most five mentioned six cycles. The second approach assumes that all the neighbors of two neighbors of $x$ are at distance at most two to $x$. With this we add a new brick to the confirmation of Conjecture \ref{conj}.


\section{Preliminaries}

We consider only simple, finite and undirected graphs. Let $G=(V,E)$ be a graph and $u,v\in V(G)$. The \emph{distance} $d(u,v)$ between $u$ and $v$ is the minimum number of edges on a path that starts in $u$ and ends in $v$ or infinite if no such path exists. The \emph{open neighborhood} $N(v)$ of $v$ is the set  $\{x\in V(G):xv\in E(G)\}$ and the \emph{closed neighborhood} $N[v]$ is the set $N(v)\cup\{v\}$. The $k$-th \emph{sphere} $S_k(v)$ of $v$ is the set $\{y\in V(G):d(y,v)=k\}$. We will mainly use the second sphere $S_2(v)$. The \emph{second open neighborhood} $N_2(v)$ of $v$ is $N(v)\cup S_2(v)$ and the \emph{second closed neighborhood} $N_2[v]$ of $v$ is $N_2(v)\cup\{v\}$. The \emph{degree} $d_G(v)$ of $v$ in $G$ is the cardinality of $N(v)$. For a positive integer $d$, $G$ is a $d$-\emph{regular graph} if $d_G(v)=d$ for every $v\in V(G)$.  For $S\subseteq V(G)$ we denote by $G[S]$ the subgraph of $G$ induced by $S$. For a positive integer $k$ we use $[k]$ for the set $\{1,\dots,k\}$. As usual, we use $C_k$ for a cycle on $k\geq 3$ vertices, in other words $k$ cycle, and $C_k$ has the length $k$. The \emph{girth} $g(G)$ of $G$ is the length of a shortest cycle in $G$ or infinite if $G$ is acyclic.  

Let $G$ be a $d$-regular graph of girth at least 5 and let $x$ be any of its vertices. Let $x_1,\ldots,x_d$ be neighbors of $x$ in $G$. Since the girth of $G$ is at least 5, all these vertices have disjoint sets of neighbors, excluding $x$. We call the set of neighbours of $x_i$, $i\in[d]$, without $x$, $X_i=\{x_i^1,\ldots,x_i^{d-1}\}$ the \emph{$i$th bunch} (with respect to $x$), i.e. $X_i=N(x_i)\backslash \{x\}$. Clearly, $S_2(x)=\bigcup_{i=1}^d X_i$. The \emph{backward degree} of $x_i^j$, with respect to an ordering of the bunches of $x$, is the number of vertices $x_p^q$, $p<i$, $q\in[d-1]$, that are adjacent to $x_i^j$. In a graph $G$ with $g(G)\geq 5$, the backward degree of $x_i^j$ is less than $i$.   

Let $G$ be a graph and let $k$ be a positive integer. A mapping $c:V(G)\rightarrow [k]$ is a (proper) $k$-\emph{coloring} of $G$ if $c(u)\neq c(v)$ for every $uv\in E(G)$. The \emph{chromatic number} $\chi(G)$ of $G$ is the least positive integer $k$ such that there exists a $k$-coloring of $G$. Let $c$ be a $k$-coloring of $G$. For $v\in V(G)$, $c(v)$ denotes the \emph{color} of $v$. By $V_i$, $i\in [k]$, we denote the $i$th \emph{color class} $\{v\in V(G):c(v)=i\}$. A color $i$ is \emph{available} for $v$ if $V_i\cap N[v]=\emptyset$. By $L(v)$ we denote the set of all available colors for $v\in V(G)$. By $\mathcal{L}_i$ we denote the family $\{L(x_i^1),\ldots,L(x_i^{d-1})\}$, $i\in[d]$. If $L(v)=\emptyset$, then $v$ is a b-\emph{vertex} (of color $c(v)$). A $k$-coloring $c$ of $G$ is a b-\emph{coloring} if every color class $V_i$, $i\in[k]$, contains a b-vertex of color $i$. The b-\emph{chromatic number} $\chi_b(G)$ is the largest integer $k$ such that there exists a b-coloring of $G$ with $k$ colors.  

\subsection{b-chromatic number of regular graphs}

In order to attack Conjecture \ref{conj} some previous results are needed. Conjecture \ref{conj} was confirmed for all graphs with girth $g(G)\geq 6$ by Kouider \cite{K2004}.

\begin{theorem}[\cite{K2004}]
Every $d$-regular graph $G$ with girth $g(G) \geq 6$ has a b-coloring with $d+1$ colors.
\end{theorem}

El Sahili and Kouider in \cite{SK2009} extended this to graphs with $g(G)\geq 5$ without cycles of length six.

\begin{theorem}[\cite{SK2009}]
\label{withoutC6}
If $G$ is a $d$-regular graph with girth $g(G) \geq 5$ and $G$ contains no $C_6$, then $\chi_b(G)=d+1$.
\end{theorem}

Conjecture \ref{conj} holds also in the case of $d$-regular graphs with large enough order. Cabelo and Jakovac \cite{CaJa} proved this with the help of Hall's marriage theorem. Their bound was improved later by El Sahili et al. \cite{SKM2015} for all $d$-regular graphs and for those without four cycles as follows. 
\begin{theorem}[\cite{SKM2015}]
If $G$ is a $d$-regular graph with at least $2d^3+2d-2d^2$ vertices, then $\chi_b(G)=d+1$.
\end{theorem}

\begin{theorem}[\cite{SKM2015}]
If $G$ is a $d$-regular graph with no $4$-cycles on at least $d^3+d$ vertices and $d\geq 7$, then $\chi_b(G) =d+1$.
\end{theorem}

Conjecture \ref{conj} was attacked and proven also for $d\leq 6$ by Blidia et al. \cite{BLIDIA20091787}.

\begin{theorem}[\cite{BLIDIA20091787}]
Let $G$ be a $d$-regular graph with girth $g(G) \geq 5$, different from the Petersen graph. If $d\leq 6$, then $\chi_b(G)=d+1$.
\end{theorem}

It follows from the above results that for any $d$, $d\geq 7$, we have only a finite number of $d$-regular graphs for which we do not know if the b-chromatic number is equal to $d+1$ or less.
Therefore, proving Conjecture \ref{conj}, we are left with the following restriction:
\begin{itemize}
    \item $G$ is $d$-regular, $d\geq 7$,
    \item its girth is exactly 5,
    \item there is a $C_6$ in $G$,
    \item $G$ has at most $2d^3+2d-2d^2-1$ vertices.
\end{itemize}

Regular graphs are in some sense easy to deal with respect to the b-chromatic number. If we can color a part of a $d$-regular graph $G$ with $d+1$ colors in such a way that every color class has a b-vertex, then we can color $G$ with $d+1$ colors. Namely, every uncolored vertex $v$ has degree $d$ and with this at most $d$ different colors in its neighborhood. This leaves at least one color for $v$. So, we can color the rest in any way, we may use the greedy algorithm, sometimes also called \emph{first fit} algorithm, to finish the coloring.

\begin{observation}\label{finish}
If we can color a part of a $d$-regular graph $G$ with $d+1$ colors in such a way that every color class has a b-vertex, then  $\chi_b(G)=d+1$.
\end{observation}

Let us recall the famous Hall's marriage theorem written in the language of transversals. 

\begin{definition}
\emph{A set $T$ is a} transversal $($a system of distinct representatives$)$ \emph{of the given family $\mathcal{A} = \{A(1);\ldots;A(s)\}$ if there is a bijection $f$ from $[s]$
onto $T$ such that $f(i)$ is an element of $A(i)$, $i \in [s]$.}
\end{definition}
 In other words, $f$ selects one representative from each set $A(i)$ in such a way that no two of these representatives are equal. 
In a famous paper of 1935 \cite{hall1987representatives}, Hall gave the first necessary and sufficient
condition for the existence of a system of distinct representatives (transversal) of a family of sets. 

\begin{theorem}\label{thm:mh}
Given a family of sets $\mathcal{A} = \{A(1);\ldots;A(s)\}$, it has a transversal if and only if for each subset of indices $I\subseteq [s]$ the following marriage condition holds:
\begin{equation}
|I|\leq \vert\bigcup_{j\in I} A(j)\vert .
\label{mhc}
\end{equation}
\end{theorem}
\noindent We will use the Hall's theorem to prove the following lemma.
 
\begin{lemma}\label{lem_general_extension}
    Let $G$ be a $d$-regular graph with $g(G) \geq 5$ and $d \geq 7$. For an arbitrary vertex $x$ there exists a proper $(d+1)$-coloring of $G$ in which $x$ and at least four of its neighbors are b-vertices.
\end{lemma}
\begin{proof}
   Let $x_1, \ldots,x_d$ be the neighbors of $x$. We color the vertex $x$ with color $d+1$ and vertices $x_i$ with color $i$, for each $i\in [d]$. Now the vertex $x$ is a b-vertex in color class $V_{d+1}$. Next, we want to show that independent of the structure of $G$ we can always color vertices in $X_1,X_2,X_3$ and $X_4$ in such a way that vertices $x_1,\ldots,x_4$ are b-vertices in the relevant color classes.

   Let us color vertices in $X_1$ arbitrarily with colors $2,3,\ldots,d$ using each color exactly once. This is always possible. Now, we will show that this coloring can always be extended for the bunches $X_2$ to $X_4$. The rest of vertices can be colored with greedy algorithm (due to Observation \ref{finish}). And in this way we will achieve the desirable proper $(d+1)$-coloring of the entire graph $G$.

   Let us assume now that vertices in bunches $X_1,\ldots, X_{t-1}$ are colored in the desirable way and now, we want to extend the coloring for vertices in $X_t$, for some $t \in \{2,3,4\}$. Note that the backward degree of each vertex in $X_t$ is at most $t-1$ and each vertex $x_t^j$ in $X_t$ has at least $d-1-(t-1)$ available colors, i.e. $|L(x_t^j)| \geq d-t$ for each $j\in [d-1]$. For the family $\mathcal{L}_t=\{L(x_t^1),\ldots,L(x_t^{d-1})\}$ the condition (\ref{mhc}) from the Hall's theorem (Theorem \ref{thm:mh}) is certainly fulfilled for each $k$ element subset when $k\leq d-t$. On the other hand, note that each color is forbidden to at most $t-1$ vertices in $X_t$. Now consider a $k$-element subset $S$ with $k \in \{d-t+1, \ldots, d-1\}$. Since $d-t+1 > t-1$ for $d\geq 7$ and $t\leq 4$, the list of all allowable colors is  $\bigcup_{x_t^j \in S} L(x_t^j) = [d]\setminus \{t\}$. Therefore $|S|\leq \bigcup_{x_t^j \in S} L(x_t^j)$ and it follows from Hall's theorem that such a coloring extension to the vertices of $X_t$ is always possible.

\end{proof}

\section{The limited number of six-cycles in $G[N_2[x]]$ containing $x$}\label{noC6}

We start with an improvement of Theorem \ref{withoutC6}. It is not necessary to forbid cycles of the length 6 in the entire graph, we  only need one vertex in a graph that is not included in any cycle $C_6$.

\begin{theorem}
Let $G$ be a $d$-regular graph, $d\geq 7$, with $g(G)=5$. If there exists a vertex in $G$ that belongs to no cycle $C_6$, then $\chi_b(G)=d+1$.
\end{theorem}

\begin{proof}
Let $G$ be a $d$-regular graph, $d\geq 7$, with $g(G)=5$ and let $x\in V(G)$ be a vertex that belongs to no cycle $C_6$. Let $x_1,\ldots,x_d$ be neighbors of $x$ in $G$. Since the girth of $G$ is 5, 
\begin{equation}
\text{all vertices $x_1,\ldots,x_d$ have disjoint sets of neighbors, excluding }x.
\end{equation}
As previously, we denote the set of neighbors of $x_i$, $i\in[d]$, limited to $S_2(x)$ (i.e. without the vertex $x$) by $X_i=\{x_i^1,\ldots,x_i^{d-1}\}$. Suppose that there exists $x_i^j$ with two neighbors in $S_2(x)$, say $x_p^q$ and $x_r^s$, $i,p,r\in [d]$, $j,q,s\in [d-1]$. If $i=p$, then $x_ix_i^jx_p^qx_i$ is a cycle of length 3, a contradiction with $g(G)=5$. So, $i\neq p$ and symetric $i\neq r$. If $p=r$, then $x_px_p^qx_i^jx_r^sx_p$ is a cycle of length 4, a contradiction with $g(G)=5$. Hence, also $p\neq r$. Now, $xx_px_p^qx_i^jx_r^sx_rx$ is a cycle of length six, a contradiction to the assumption that $x$ belong to no $C_6$. Therefore we have     
\begin{equation}
  d_{G[S_2(x)]}(x_i^j)\leq 1 \text{ and } d_{G[X_i]}(x_i^j)=0\label{cond:atmost1}
\end{equation}
for every $i\in [d]$ and $j\in [d-1]$. In other words, condition (\ref{cond:atmost1}) states: $|N(x^j_i)\cap \bigcup_{k\not=i}X_k|\leq 1$ and $|N(x^j_i)\cap X_i|=0$ for each $i\in[d]$ and $j\in [d-1]$.
The two above mentioned conditions imply that the edge set of the graph induced by $X_i\cup X_j$ is a matching, for any $i,j \in[d]$. 

Our aim is to achieve a b-coloring with $d+1$ colors, given by color classes $(V_1,\ldots,V_{d+1})$.
We color $x$ with $d+1$ and $x_i$ with color $i$, $i\in[d]$. So, $x$ is a b-vertex for color class $V_{d+1}$. Now, we will color $d-1$ vertices in each $X_i$ with different colors from $[d]\backslash \{i\}$ to make each vertex $x_i$ a b-vertex for $V_i$, $i\in[d]$. After that, $\chi_b(G)=d+1$ follows from Observation \ref{finish}.

We start with $X_1$ and color the vertices arbitrary with different colors from $\{2,\ldots,d\}$. Suppose that we have colored vertices in $X_1\cup \cdots \cup X_t$ in such a way that $x_i$ is a b-vertex in $V_i$ and the partial coloring of $G$ is proper for some $t\in [d-1]$. Now, we want to assign colors to the vertices from $X_{t+1}$. If the proper extension of the coloring such that $x_{t+1}$ is a b-vertex for $V_{t+1}$ is possible, then we are done. So, let us assume that each assignment of colors from $[d]\backslash \{t+1\}$ to vertices from $X_{t+1}$ yields a monochromatic edge. We choose such a coloring of $X_{t+1}$ that minimizes the number of monochromatic edges between $X_1\cup\cdots\cup X_t$ and $X_{t+1}$. Let $x_i^jx_{t+1}^s$ be a monochromatic edge, for some $i\in [t]$ and some $j,s\in [d-1]$, colored with $k$. If there exists a vertex $x_{t+1}^p$, $p\in [d-1]\backslash\{s\}$ without a neighbor in $X_1\cup\cdots\cup X_t$, then we can swap the colors of $x_{t+1}^p$ and $x_{t+1}^s$ to obtain a coloring with one monochromatic edge less, a contradiction of the minimality of monochromatic edges. Hence, every $x_{t+1}^p$, $p\in [d-1]$ has a neighbor in $X_1\cup\cdots\cup X_t$. If there exists $x_{t+1}^p$, $p\in [d-1]\backslash\{s\}$ with a neighbor in $X_i$, then we swap the colors of $x_{t+1}^p$ and $x_{t+1}^s$. It follows from (\ref{cond:atmost1}) that the new coloring has at least one monochromatic edge less, the same contradiction again. Thus, we can assume that $x_{t+1}^s$ is the only vertex with a neighbor in $X_i$. If $t<d-1$, then there exists $x_{t+1}^p$ and $x_{t+1}^r$, $p,r\in[t]\backslash\{s\}$ both with their neighbors in $X_{\ell}$ for some $\ell\in[t]\backslash\{i\}$. These two neighbors are different by (\ref{cond:atmost1}) and therefore of different colors. At least one of them is different from $k$ and we may assume that the neighbor of $x_{t+1}^p$ is colored differently than $k$. Now, we swap colors of $x_{t+1}^p$ and $x_{t+1}^s$ to get to the same contradiction. If $t=d-1$, then $k<d$. Moreover, each of $X_{t+1}\backslash\{x_{t+1}^s\}$ is adjacent to exactly one vertex in $X_1\cup\cdots\cup X_t$ and no two different vertices from $X_{t+1}$ have a neighbor in the same $X_q$, $q\in[d-1]$. But then there exists a vertex $x_{t+1}^p$ with a neighbor in $X_k$. No vertex of $X_k$ is  colored with $k$ and we can swap the colors of $x_{t+1}^p$ and $x_{t+1}^s$ to obtain a contradiction with the minimality of monochromatic edges again. Hence, vertices of $X_{t+1}$ can be colored without the appearance of monochromatic edges. 
\end{proof}

Now, we will consider the case when the number of cycles $C_6$ that any vertex $x$ belongs to is not very large. We can bound the number of cycles $C_6$ in $G[N_2[x]]$ containing $x$  to ensure that Conjecture \ref{conj} holds. We start with a simple observation and an auxiliary lemma.

\begin{observation}\label{obs:degree}
    Let $G$ be a $d$-regular graph, $d \geq 7$, with $g(G)=5$, $x\in V(G)$ and $x_i^j\in X_i$ for some $i\in [d]$ and $j\in[d-1]$. If $d_{G[S_2(x)]}(x_i^j)=p>1$, then there exists exactly $\binom{p}{2}$ cycles $C_6$ in $G[N_{2}[x]]$ containing $x$ and $x_i^j$ as vertices at distance 3 in these cycles.
\end{observation}

\begin{lemma}\label{lm:bdeg1}
    Let $G$ be a $d$-regular graph, $d \geq 7$, with the girth 5 and let $x$ be any arbitrary vertex of $G$. If every vertex of $X_i$, $2 \leq i \leq d$ has a backward degree at most 1 and $\bigcup_{t \in [i-1]}X_t \cup N[x]$ is partially $(d+1)$-colored in such a way that $x,x_1,\ldots,x_{i-1}$ are b-vertices in different color classes $V_{d+1},V_1, \ldots, V_{i-1}$, respectively, then the coloring can be extended into $X_i$ in such a way that the vertex $x_i$ is a b-vertex in $V_i$.
\end{lemma}
\begin{proof}
    Note that if a vertex $v$ in $X_i$, for any $2\leq i\leq d$, has backward degree equals to 1, then at most one color from $[d]\setminus\{i\}$ is
forbidden for this vertex. It means that $|L(v)|\geq d-2$ for such a vertex $v$. Hence, for each subset $S$ of $X_i$ of cardinality at most $d-2$ we have $|\bigcup_{v\in S}L(v)|\geq |S|$ and we are done by Hall's theorem.
In addition, independently on the backward degree of a vertex, each color from $[i-1]$ is forbidden for at most $i-2$ vertices in $X_i$. If $|S|=d-1$, i.e. $S=X_i$, then $\bigcup_{v\in S}L(v)$ contains all colors from $[d]\setminus \{i\}$. So, by Hall's theorem we can color every vertex in $X_i$, $2\leq i \leq d$.
\end{proof}

\begin{theorem}\label{thm:5C6}
Let $G$ be a $d$-regular graph, $d\geq 7$, with $g(G)=5$. If there exists a vertex $x$ in $G$ contained in at most 5 cycles $C_6$ in $G[N_2[x]]$, then $\chi_b(G)=d+1$.
\end{theorem}
\begin{proof}
Let $x\in V(G)$ be a vertex included in at most 5 cycles $C_6$ in $G[N_2[x]]$. 
For each of its neighbor $s$ we establish the degree sequence of $N(s)\backslash \{x\}$ limited to $G[S_2(x)]$ and order its entries in a non-ascending way. Next, we order all the sequences for $x$ in the reverse lexicographic order and due to this order we number relevant neighbors of $x$ as $x_1,\ldots,x_d$, respectively.

Since $x$ belongs to at most 5 cycles $C_6$ in $G[N_2[x]]$ and by Observation \ref{obs:degree}, we have $d_{G[S_2(x)]}(x_i^j)\leq 3$ for every $i\in[d]$ and $j\in[d-1]$. Consequently, one of the following situations occur:
$(1)$ at most 5 vertices from $S_2(x)$ has degree 2 in $G[S_2(x)]$ and the rest of vertices from $S_2(x)$ has degree at most one in $G[S_2(x)]$, or $(2)$ at most one vertex from $S_2(x)$ has degree three in $G[S_2(x)]$, at most two vertices from $S_2(x)$ has degree two in $G[S_2(x)]$, and the rest of the vertices from $S_2(x)$ has degree at most one in $G[S_2(x)]$.
Note that, due to the ordering of the vertices, vertices of degree greater than 1 in $G[S_2(x)]$ may only appear in the first 5 bunches in case $(1)$ or the first 3 bunches in case $(2)$. 

We start with case $(1)$. Note that since $g(G)=5$ and $d\geq 7$, by Lemma~\ref{lem_general_extension} we can color  vertices from $\bigcup_{1\leq i\leq 4}X_i\cup N[x]$ such that $x$, $x_1$, $x_2$, $x_3$ and $x_4$ are b-vertices: $c(x):=d+1$, $c(x_i):=i$, $i\in [d]$ (due to the proof of Lemma~\ref{lem_general_extension}).
Now we are in a position to extend this coloring such that $x_5$, $x_6,x_7,\ldots ,x_d$ are also b-vertices. We need to color vertices in each $X_i$, $5\leq i\leq d$, with colors $[d]\setminus\{i\}$.
First, we color vertices from $X_5$. When we take into account only backward degree of the vertices from $\bigcup_{5\leq i\leq d}X_i$, only one vertex $u$ in $X_5$ can have backward degree equal to 2 and any other vertex from this set has backward degree at most 1 in $G[S_2(x)]$.
If this happens, then $|L(u)|\geq d-3$ and $|L(v)|\geq d-2$ for any $v\in X_5\setminus \{u\}$. Hence, for each subset $S$ of $X_5$ of cardinality at most $d-3$ we have $|\bigcup_{v\in S}L(v)|\geq |S|$ and we are done by Hall's theorem. If $d-2\leq |S| \leq d-1$, then each color from $[d]\setminus \{5\}$ is forbidden for at most 4 vertices in $X_5$. Since $d-2\geq 5$, each color from $[d]\setminus \{5\}$ appears in $\bigcup_{x\in S}L(x)$, which implies that $|\bigcup_{x\in S}L(x)|=d-1$. Hence it is possible to properly color the vertices from $X_5$ by Hall's theorem. 
If we do not have such a vertex $u$ of backward degree 2 in $G[S_2(x)]$, then every vertex in $X_5$ has backward degree at most 1 and it follows from Lemma \ref{lm:bdeg1} that we can colour $X_5$ in such a way that $x_5$ is a b-vertex.
In the uncolored bunches $X_6, \ldots, X_d$ every vertex has backward degree at most 1 and so, by 
Lemma \ref{lm:bdeg1}, we can color every vertex in $X_i$, $2\leq i \leq d$.

Now consider case $(2)$. Note that in this case at most 3 bunches contains vertices of degree greater than one limited to $G[S_2(x)]$. So, first we apply Lemma~\ref{lem_general_extension} to receive an appropriate coloring of the  vertices from $\bigcup_{1\leq i\leq 4}X_i\cup N[x]$ and next we extend it due to Lemma \ref{lm:bdeg1} for the rest of the vertices in the subgraph $G[N_2[x]]$. 

Finally, in each case, by Observation \ref{finish} we can extend obtained coloring of $G[N_2[x]]$ for the entire graph $G$. Thus $\chi_b(G)=d+1$.
\end{proof}


\section{Two bunches without neighbors at distance three from $x$}\label{twoBunchd}

In this section we consider $d$-regular graphs, $d\geq 7$, such that for some vertex $x$, at least two of its bunches have all their neighbors within $N_2(x)$.

\begin{theorem}\label{thm:d:2bunches}
Let $G$ be a $d$-regular graph, $d\geq 7$, of the girth 5. If there exists a vertex $x$ in $V(G)$ with at least two bunches where all the neighbors of vertices in a bunch are in $N_2(x)$, then $\chi_b(G)=d+1$. 
\end{theorem}

\begin{proof}
Let $x$ be a vertex of $G$ fulfilling the assumption about its bunches. Let $X_1$ and $X_d$ be the bunches with all their neighbors belonging to $N_2(x)$.

Note that the rest of the bunches for $x$ has not been numbered yet. In addition, the vertices within each bunch are also unordered. Our aim is to order the bunches (columns) and vertices within each bunch (rows) in such a way that the requirements, given in details below, hold. For the reader's convenience we number bunches/columns from left to right while vertices in bunches/rows are numbered from top to bottom (similarity to the matrix entries numbering - cf. Fig.~\ref{ex:matrix}).

\begin{description}
    \item[Requirement 1. ] All vertices in $X_1$ and $X_d$ have all their neighbors in $N_2(x)$ - the bunches from the assumption of the theorem.
    \item[Requirement 2. ] $N(x_1^j)=\{x_1,x_i^j: i\in\{2,3,\ldots,d\}\}$ for each $j\in[d-1]$, i.e. each row is formed by the closed neighborhood, excluding $x_1$, of a vertex from the first bunch. 
    \item[Requirement 3. ] $d-3$ vertices forming the set $N(x_d^{j})\cap X_{j+1}$, $j\in [d-3]$, are lying in the first three rows and they form an independent set, denoted by $\mathbb{I}_1$.

    \item[Requirement 4. ] The vertex from $N(x_d^{d-2})\cap X_{d-1}$ is lying in the second row.
\end{description}

Our aim is to achieve a partial $(d+1)$-coloring $c$ of $N_2[x]$, with color classes $V_1,\ldots,V_{d+1}$, that $x$ is a b-vertex for the color class $V_{d+1}$, and we have $d-2$ b-vertices among neighbors of $x$ and two b-vertices in $X_1$ - this is a novel approach.
This coloring when the desired ordering of $N_2(x)$ due to Requirements 1-4 is given, is defined as follows:
\begin{enumerate}
    \item $c(x):=d+1$,
    \item $c(x_i):=i$ for each $i\in [d]$,
    \item $c(x_1^j):=j+1$, $j\in[d-1]$,
    \item $c(v):=1$, for each $v \in \mathbb{I}_1$,
    \item $c(x_d^j):=d+1$ for each $j\in \{4, \ldots, d-1\}$,\label{step:5}
    \item for every uncolored vertex $w$ from the first $d-2$ columns and for vertices $w$ lying in $X_{d-1}$ in the last $d-4$ rows (i.e. all uncolored vertices excluding 6 vertices $x_{j}^t$, $j\in \{d-1,d\}$ and $t\in [3]$) set $c(w)=c(x_1^s)$ where $wx_d^s\in E(G)$ for some $s\in[d-1]$, \label{step:6}
    \item the remaining 6 uncolored vertices (fields with bold frames in top-right corner of matrix in Fig.~\ref{ex:matrix} or~\ref{ex:matrix2}) are colored in the greedy way.\label{step:7}
\end{enumerate}

Before we show that such a coloring of $N_2[x]$ is well defined, we give several straightforward observations based on Requirements 1-4 and the assumption of the theorem.

\begin{observation}\label{obs:reqc}
Each vertex $x_j^i$ in $j$-th column and $i$-th row is adjacent to at most one vertex in each other column for any $i\in[d-1]$, $j\in[d]$. Moreover, each vertex $x_j^i$ in $j$-th column and $i$-th row is adjacent to at most one vertex in each other row for every $i\in[d-1]$, $j \in \{2,\ldots, d\}$. 
\end{observation}

\begin{observation}\label{obs:reqe}
Each vertex $x_1^i$, $i \in [d-1]$, in the first bunch $X_1$ has all its neighbors in the same $i$-th row and this means exactly one neighbor in each other bunch/column.
\end{observation}

\begin{observation}\label{obs:reqf}
Each vertex $x_d^i$, $i \in [d-1]$, in the last bunch $X_d$ has exactly one neighbor in each other row and each other bunch/column.
\end{observation}

\begin{observation}\label{obs:reqd}
Establishing a vertex in any row in one of the bunches determines the vertices in the same row in all the other bunches.
\end{observation}

The first observation  is due to the girth assumption, $g(G)=5$, because otherwise there would exist $C_4$. The second and the third observations follow from the fact that bunches $X_1$ and $X_d$ have all the neighbors in $N_2(x)$ and by Observation \ref{obs:reqc}. The last observation follows from Requirement 2.

Now, we can show that such a coloring of $N_2[x]$ is proper. A vertex $w\in S_2(x)\backslash X_d$ has exactly one neighbor in $X_d$ by Observation \ref{obs:reqc}.
Note that all the colored vertices from $X_1,\ldots ,X_{d-2}$ have an unique color within their bunch. 
It implies also that vertices $x_i$, $i \in [d-2]$ are b-vertices in their color classes. 
Set $\mathbb{I}_1$ and with this also $V_1$ is independent. 
Suppose now that $u$ and $v$ from $S_2(x)\setminus X_d$ are colored with the same color $c(u)=c(v)$ for some $i\in[d-1]$. Due to the Step \ref{step:6} of the coloring procedure, vertices that have  received the same color must be adjacent to the same vertex in the last bunch $X_d$. So, if  $uv$ is an edge, then we have a triangle $x_d^tuvx_d^t$, where $x_d^t$ is the neighbor of $u$ and $v$, which is not possible. Therefore, the color classes $V_2\dots,V_d$ represents independent sets in $G$. Of course, the color class $V_{d+1}$ is also an independent set due to the coloring procedure (cf. Steps \ref{step:5} and \ref{step:7}) and we have a proper coloring of $N_2[x]$.

In addition, the coloring guarantees us the desirable b-vertices in all color classes, due to our requirements. Indeed, the set of b-vertices of the partial coloring includes vertices $x,x_1,x_2,\ldots,x_{d-2}$ which has been shown already. In addition vertices $x_1^{d-2}$ and $x_1^{d-1}$, similarly as vertices $x_1^4,\ldots,x_1^{d-3}$, are also b-vertices. It is enough to notice that all the vertices in the relevant rows are colored and the colors are unique within the row.

The rest of the graph $G$ is colored in the greedy way. So, $\chi_b(G)=d+1$ by Observation \ref{finish}.


Now we show how the assumed ordering of the vertices in rows and columns that meets the desired Requirements 1-4 can be achieved. Clearly, Requirement 1 is satisfied by the assumption. To satisfy Requirement 2, whenever a vertex $x_i^j$ is determined, $i\neq 1$, we define its neighbor in $X_1$ - the vertex $x_1^j$. As soon as vertex $x_1^j$ is fixed we can determine the vertices adjacent to it in the remaining bunches - the entire row $j$ is established. 
So, we need only to take care about Requirements 3 and 4. 
We start with $\mathbb{I}_1:=\emptyset$. Let $X_1$ and $X_d$  be the bunches from the assumption of the theorem - the bunches with all their neighbors in $N_2(x)$. In addition, at the beginning we have $d-2$ unordered bunches/columns with $d-1$ unordered rows. The further part of the construction is as follows.

\begin{itemize}
    \item Choose any vertex from $X_d$ and label it as $x_d^1$ - this is a vertex in the first row and in the last column of $S_2(x)$. With this all vertices of the first row (in ordered and unordered bunches) are established by Observation \ref{obs:reqd}.  
    \item Choose any unordered bunch as $X_2$, fix for $x_2^2$ the vertex from $X_2$ that is adjacent to $x_d^1$ and set $\mathbb{I}_{1} \leftarrow x_2^2$ (Requirement 3 is met for $j=1$). All the vertices of the 2nd row are now determined in each other bunch by Observation \ref{obs:reqd}. 

    \item   
    To choose an unordered bunch as $X_3$ we do in such a way that the following conditions will hold simultanously:
    \begin{itemize}
        \item $x_d^2x_3^3\in E(G)$,
        \item $x_d^3x_2^1 \notin E(G)$ (it will be used in Subcase 2 in the further part of the ordering).
    \end{itemize}   
    Such a choice of $X_3$ is possible because we need to choose such a bunch among $d-3$, $d\geq 7$ unordered bunches. Only at most one bunch does not keep the requirement (by Observation \ref{obs:reqc}). 

   Let $\mathbb{I}_{1} \leftarrow x_3^3$ (the first part of Requirement 3 folds for $j=2$, i.e. the vertex from $N(x_d^2)\cap X_{3}$ is lying in the first three rows).
 
    \item Choosing a bunch as $X_4$ can be two-fold. In this point we will also choose the bunch $X_{d-1}$
    to ensure that Requirement 4 holds.
    
    \begin{description}
    \item[Subcase 1:] The vertex from $N(x_d^3)$ lying in the second row is in one of the $d-4$ bunches  that are not yet ordered.

    We select exactly this bunch as $X_4$, i.e. $x_4^2x_d^3\in E(G)$, and let $\mathbb{I}_{1} \leftarrow x_4^2$ (cf. Fig.~\ref{ex:matrix}).
    Up to now, $\mathbb{I}_1=\{x_2^2,x_3^3,x_4^2\}$ and now we show that $\mathbb{I}_1$ is an independent set. The vertex $x_3^3$ was chosen as the neighbor of $x_d^2$ in $X_3$, so this vertex already has its neighbor in the second row. Hence $x_2^2x_3^3,x_3^3x_4^2\notin E(G)$ by Observation \ref{obs:reqc}. 
    In addition, $x_2^2$ and $x_4^2$ are not adjacent because they both lie in the same row and they are both adjacent to $x_1^2$. Thus, up to now, the set $\mathbb{I}_1$ is independent and contains vertices from the first three rows. Thus, Requirement 3 is satisfied for $j\in[3]$.
     
    To designate $X_{d-1}$ in this subcase we note that all the vertices from the second row in the bunches $X_1,X_2$ and $X_4$, have their neighbor in $X_d$ among $x_d^j$, $j\in[3]$. So, the vertex from $N(x_3^2)\cap X_d$ is among unordered vertices in $X_d$. 
    Let $x_d^{d-1}$ be fixed by being a neighbor of $x_3^2$. With this the last row in every bunch is determined by Observation \ref{obs:reqd}. 
    Now, we can take any vertex among the unordered vertices from $X_d$ and fix it as $x_d^{d-2}$. Note that the vertex from $N(x_d^{d-2})\cap R_2$ is among the unordered bunches. We take this bunch as $X_{d-1}$ and the Requirement 4 holds.

\begin{figure}[ht]
\begin{center}
\begin{pspicture}(0,1)(11,10)

\psframe[shadow=true,linewidth=2.5pt](1,1)(10,9)

\psline[linewidth=1pt](2,1)(2,9)
\psline[linewidth=1pt](3,1)(3,9)
\psline[linewidth=1pt](4,1)(4,9)
\psline[linewidth=1pt](5,1)(5,9)
\psline[linewidth=1pt](7,1)(7,9)
\psline[linewidth=1pt](8,1)(8,9)
\psline[linewidth=1pt](9,1)(9,9)

\psline[linewidth=1pt](1,2)(10,2)
\psline[linewidth=1pt](1,3)(10,3)
\psline[linewidth=1pt](1,4)(10,4)
\psline[linewidth=1pt](1,6)(10,6)
\psline[linewidth=1pt](1,7)(10,7)
\psline[linewidth=1pt](1,8)(10,8)

\psline[linewidth=3pt](8,8)(10,8)
\psline[linewidth=3pt](8,7)(10,7)
\psline[linewidth=3pt](8,6)(10,6)

\psline[linewidth=3pt](8,6)(8,8.95)
\psline[linewidth=3pt](9,6)(9,8.95)

\psframe[fillstyle=hlines,hatchangle=90,hatchsep=6pt](9,8)(10,9)
\psframe[fillstyle=vlines,hatchangle=90,hatchsep=4pt](9,7)(10,8)
\psframe[fillstyle=crosshatch,hatchangle=90,hatchsep=6pt](9,6)(10,7)

\psframe[fillstyle=hlines,hatchangle=90,hatchsep=6pt](2,7)(3,8)
\psframe[fillstyle=vlines,hatchangle=90,hatchsep=4pt](3,6)(4,7)
\psframe[fillstyle=crosshatch,hatchangle=90,hatchsep=6pt](4,7)(5,8)

\psframe[fillstyle=vlines,hatchsep=6pt](9,1)(10,2)
\psframe[fillstyle=hlines,hatchsep=6pt](9,2)(10,3)
\psframe[fillstyle=crosshatch,hatchsep=6pt](9,3)(10,4)

\psframe[fillstyle=crosshatch,hatchsep=6pt](7,7)(8,8)
\psframe[fillstyle=vlines,hatchsep=6pt](3,7)(4,8)
\psframe[fillstyle=hlines,hatchsep=6pt](8,7)(9,8)

\psframe[fillstyle=solid,fillcolor=white,linecolor=white](2.4,7.3)(2.6,7.7)
\psframe[fillstyle=solid,fillcolor=white,linecolor=white](4.4,7.3)(4.6,7.7)
\psframe[fillstyle=solid,fillcolor=white,linecolor=white](3.4,6.3)(3.6,6.7)
\psframe[fillstyle=solid,fillcolor=white,linecolor=white](7.4,7.3)(7.6,7.7)

\psframe[fillstyle=solid,fillcolor=white,linecolor=white](9.1,3.35)(9.85,3.7)
\psframe[fillstyle=solid,fillcolor=white,linecolor=white](9.1,2.35)(9.85,2.7)
\psframe[fillstyle=solid,fillcolor=white,linecolor=white](9.1,1.35)(9.85,1.7)

\rput(1.5,9.5){$X_1$}
\rput(2.5,9.5){$X_2$}
\rput(3.5,9.5){$X_3$}
\rput(4.5,9.5){$X_4$}
\rput(6,9.5){$\hdots$}
\rput(7.5,9.5){$X_{d-2}$}
\rput(8.5,9.5){$X_{d-1}$}
\rput(9.5,9.5){$X_d$}

\rput(0.5,8.5){$R_1$}
\rput(0.5,7.5){$R_2$}
\rput(0.5,6.5){$R_3$}
\rput(0.5,5){$\vdots$}
\rput(0.5,3.5){$R_{d-3}$}
\rput(0.5,2.5){$R_{d-2}$}
\rput(0.5,1.5){$R_{d-1}$}

\rput(9.5,5){\textbf{d+1}}
\rput(9.5,3.5){\textbf{d+1}}
\rput(9.5,2.5){\textbf{d+1}}
\rput(9.5,1.5){\textbf{d+1}}

\rput(1.5,8.5){\textbf{2}}
\rput(1.5,7.5){\textbf{3}}
\rput(1.5,6.5){\textbf{4}}
\rput(1.5,5){$\vdots$}
\rput(9.5,4.5){$\vdots$}
\rput(9.5,5.7){$\vdots$}
\rput(1.5,3.5){\textbf{d-2}}
\rput(1.5,2.5){\textbf{d-1}}
\rput(1.5,1.5){\textbf{d}}

\rput(5.5,7.5){$\hdots$}
\rput(6.5,7.5){$\hdots$}
\rput(2.5,7.5){\textbf{1}}
\rput(3.5,6.5){\textbf{1}}
\rput(4.5,7.5){\textbf{1}}
\rput(6,7.5){\textbf{1}}
\rput(7.5,7.5){\textbf{1}}

\end{pspicture}
\caption{Matrix $M_{d-1,d}$ corresponding to the performed ordering including Subcase $1$. Each integer in field $m_{ij}$ of $M$ represents the color dedicated to the vertex laying in row $R_i$ and bunch $X_j$, i.e. $c(x_j^i)$. In addition, if two fields $m_{ij}, m_{kl}$ are filled with the same pattern then $x_j^ix_l^k \in E(G)$.}
\label{ex:matrix}
\end{center}
\end{figure}

    \item[Subcase 2:] The vertex from $N(x_d^3)$ lying in the second row is among already ordered bunches $X_1,X_2,X_3$ or $X_d$. 
    
    Note that the vertex from $N(x_d^3) \cap R_2$ is not in $X_1$ neither in $X_2$ by Observation \ref{obs:reqc}. Clearly, $x_d^3$ is not adjacent to $x_d^2$ and therefore $x_d^3x_3^2\in E(G)$.
    By Observation \ref{obs:reqc} $x_d^3$ is not adjacent to $x_1^1$ neither to $x_3^1$ and not to $x_2^1$ by our assumption when fixing $X_3$. Hence the neighbor of $x_d^3$ lying in the first row, i.e. the vertex from $N(x_d^3)\cap R_1$ is  within unordered bunches. So, we fix $X_4$ such one that $x_4^1$ is adjacent to $x_d^3$ and $\mathbb{I}_{1} \leftarrow x_4^1$ (cf. Fig.~\ref{ex:matrix2}). 
    
    Up to now, $\mathbb{I}_1=\{x_2^2,x_3^3,x_4^1\}$ and we show that $\mathbb{I}_1$ is independent at this stage. As before $x_3^3x_2^2\notin E(G)$ by Observation \ref{obs:reqc} since $x_3^3x_d^2\in E(G)$. Similarly the edge $x_2^2x_d^1\in E(G)$ implies that $x_2^2x_4^1\notin E(G)$ by Observation \ref{obs:reqc}.  Moreover, $x_4^1$ was chosen as the neighbor of $x_d^3$, so it does not have any other neighbor in the third row, i.e. $x_4^1x_3^3\notin E(G)$ by Observation \ref{obs:reqc}.
    
    To designate $X_{d-1}$ we note that all the vertices from the second row in $X_1,X_2$ and $X_3$ have their neighbors in $X_d$ among $x_d^j$, $j\in[3]$. We fix $x_d^{d-1}$ as the neighbor of $x_4^2$ in $X_d$. With this the last row in every bunch is determined by Observation \ref{obs:reqd}. 
    Next, we are also interested in the neighbor of $x_4^1$ from the second row, i.e. from $N(x_4^1)\cap R_2$. First note that such a neighbor may not exist. If this is the case, then, similarly as in Subcase 1 we take any vertex among the unordered vertices from $X_d$ and fix it as $x_d^{d-2}$. Note that the vertex from $N(x_d^{d-2})\cap R_2$ is among unordered bunches. We take this bunch as $X_{d-1}$ and the Requirement 4 holds. Otherwise, if such a neighbor exists, then certainly it must be in the unordered bunches. Indeed, $x_1^2x_4^1\notin E(G)$ by Observation \ref{obs:reqc} because $x_4^1$ has already its neighbor in $X_1$. Similarly $x_2^2x_4^1\notin E(G)$ by Observation \ref{obs:reqc} because $x_2^2$ already has its neighbor $x_d^1$ in the first row. In addition, $x_3^2x_4^1\notin E(G)$ because together with $x_d^3$ such an edge would form $C_3$. Let $X_{d-1}$ be the bunch such that $x_{d-1}^2x_4^1\in E(G)$ (this property will be used in the further part to prove that $\mathbb{I}_1$ is independent). Moreover, we fix $x_d^{d-2}$ to be the neighbor of $x_{d-1}^2$ and the row before the last one is also determined in every bunch by Observation \ref{obs:reqd}.
    \end{description}
    
\begin{figure}[ht]
\begin{center}
\begin{pspicture}(0,1)(11,10)

\psframe[shadow=true,linewidth=2.5pt](1,1)(10,9)

\psline[linewidth=1pt](2,1)(2,9)
\psline[linewidth=1pt](3,1)(3,9)
\psline[linewidth=1pt](4,1)(4,9)
\psline[linewidth=1pt](5,1)(5,9)
\psline[linewidth=1pt](7,1)(7,9)
\psline[linewidth=1pt](8,1)(8,9)
\psline[linewidth=1pt](9,1)(9,9)

\psline[linewidth=1pt](1,2)(10,2)
\psline[linewidth=1pt](1,3)(10,3)
\psline[linewidth=1pt](1,4)(10,4)
\psline[linewidth=1pt](1,6)(10,6)
\psline[linewidth=1pt](1,7)(10,7)
\psline[linewidth=1pt](1,8)(10,8)

\psline[linewidth=3pt](8,8)(10,8)
\psline[linewidth=3pt](8,7)(10,7)
\psline[linewidth=3pt](8,6)(10,6)

\psline[linewidth=3pt](8,6)(8,8.95)
\psline[linewidth=3pt](9,6)(9,8.95)

\psframe[fillstyle=vlines,hatchangle=90,hatchsep=6pt](9,8)(10,9)
\psframe[fillstyle=hlines,hatchangle=90,hatchsep=4pt](9,7)(10,8)

\pspolygon[fillstyle=crosshatch,hatchangle=90,hatchsep=6pt](9,6)(9,7)(10,7)(9,6)
\pspolygon[fillstyle=dots,hatchangle=90,hatchsep=4pt](9,6)(10,6)(10,7)(9,6)

\psframe[fillstyle=vlines,hatchangle=90,hatchsep=6pt](2,7)(3,8)
\psframe[fillstyle=hlines,hatchangle=90,hatchsep=4pt](3,6)(4,7)
\psframe[fillstyle=crosshatch,hatchangle=90,hatchsep=6pt](4,8)(5,9)

\psframe[fillstyle=vlines,hatchsep=6pt](9,1)(10,2)
\psframe[fillstyle=hlines,hatchsep=6pt](9,2)(10,3)
\psframe[fillstyle=crosshatch,hatchsep=6pt](9,3)(10,4)

\psframe[fillstyle=crosshatch,hatchsep=6pt](7,7)(8,8)
\psframe[fillstyle=vlines,hatchsep=6pt](4,7)(5,8)
\psframe[fillstyle=hlines,hatchsep=6pt](8,7)(9,8)
\psframe[fillstyle=dots,hatchsep=4pt](3,7)(4,8)

\psframe[fillstyle=solid,fillcolor=white,linecolor=white](2.4,7.3)(2.6,7.7)
\psframe[fillstyle=solid,fillcolor=white,linecolor=white](4.4,8.3)(4.6,8.7)
\psframe[fillstyle=solid,fillcolor=white,linecolor=white](3.4,6.3)(3.6,6.7)
\psframe[fillstyle=solid,fillcolor=white,linecolor=white](7.4,7.3)(7.6,7.7)

\psframe[fillstyle=solid,fillcolor=white,linecolor=white](9.1,3.35)(9.85,3.7)
\psframe[fillstyle=solid,fillcolor=white,linecolor=white](9.1,2.35)(9.85,2.7)
\psframe[fillstyle=solid,fillcolor=white,linecolor=white](9.1,1.35)(9.85,1.7)

\rput(1.5,9.5){$X_1$}
\rput(2.5,9.5){$X_2$}
\rput(3.5,9.5){$X_3$}
\rput(4.5,9.5){$X_4$}
\rput(6,9.5){$\hdots$}
\rput(7.5,9.5){$X_{d-2}$}
\rput(8.5,9.5){$X_{d-1}$}
\rput(9.5,9.5){$X_d$}

\rput(0.5,8.5){$R_1$}
\rput(0.5,7.5){$R_2$}
\rput(0.5,6.5){$R_3$}
\rput(0.5,5){$\vdots$}
\rput(0.5,3.5){$R_{d-3}$}
\rput(0.5,2.5){$R_{d-2}$}
\rput(0.5,1.5){$R_{d-1}$}

\rput(9.5,5){\textbf{d+1}}
\rput(9.5,3.5){\textbf{d+1}}
\rput(9.5,2.5){\textbf{d+1}}
\rput(9.5,1.5){\textbf{d+1}}

\rput(1.5,8.5){\textbf{2}}
\rput(1.5,7.5){\textbf{3}}
\rput(1.5,6.5){\textbf{4}}
\rput(1.5,5){$\vdots$}
\rput(9.5,4.5){$\vdots$}
\rput(9.5,5.7){$\vdots$}
\rput(1.5,3.5){\textbf{d-2}}
\rput(1.5,2.5){\textbf{d-1}}
\rput(1.5,1.5){\textbf{d}}

\rput(5.5,7.5){$\hdots$}
\rput(6.5,7.5){$\hdots$}
\rput(2.5,7.5){\textbf{1}}
\rput(3.5,6.5){\textbf{1}}
\rput(4.5,8.5){\textbf{1}}
\rput(6,7.5){\textbf{1}}
\rput(7.5,7.5){\textbf{1}}

\end{pspicture}
\caption{Matrix $M_{d-1,d}$ corresponding to the performed ordering including Subcase $2$. Each integer in field $m_{ij}$ of $M$ represents color dedicated to vertex laying in row $R_i$ and bunch $X_j$, i.e. $c(x_j^i)$. In addition, if two fields $m_{ij}, m_{kl}$ are filled with the same pattern then $x_j^ix_l^k \in E(G)$.}
\label{ex:matrix2}
\end{center}
\end{figure}

Note that up to now we have established bunches $X_j$ for $j\in[4]\cup\{d-1,d\}$ as well as rows $R_i$ for $i\in[3]\cup \{d-2,d-1\}$. 

\item In this step we complete the ordering.
    
    Note that the unordered vertices of $X_d$ have a neighbor in the second row among the unordered bunches. 
    We arbitrarily order the remaining vertices of $X_d$.
    For every $j\in \{4,\dots,d-3\}$ we choose a bunch $X_{j+1}$ to be such that $x_d^jx_{j+1}^2\in E(G)$. After this all the vertices are ordered with respect to bunches as well as with respect to rows. Moreover, Requirement 4 holds. 
        
    Finally, let $\mathbb{I}_1:= \mathbb{I}_1 \cup (\bigcup_{i=5}^{d-2} \{x_i^2\})$. Note that $\mathbb{I}_1$ is still an independent set. Indeed, to the previous $\mathbb{I}_1$ of cardinality 3 we have added vertices from the same row that are definitely not adjacent to each other, and additionally, by selecting $X_{d-1}$ in Subcase 2 (i.e. by choosing the vertex from $N(x_4^1)\cap R_2$ as $x_{d-1}^2$), they are not adjacent to the previously selected vertices in $\mathbb{I}_1$.
    \end{itemize}
    
    
    Hence, the desirable ordering and, in consequence, the desirable coloring is achievable. The proof is complete.
\end{proof}

\section{Conclusion and future work}

Two approaches presented in this work are kind of complementary with the respect to the density (number) of edges in $G[S_2(x)]$. First approach with the limited number of six-cycles including $x$ in $N_2[x]$ yield a small number of edges in $G[S_2(x)]$, while the other approach, when (at least) two bunches have all of its neighbors in $N_2(x)$ yields much bigger number of edges in $G[S_2(x)]$. So, one could hope that by improving both approaches, one could meet somewhere in the middle and confirm Conjecture \ref{conj}.

For the first approach, one could hope to improve the fixed number (five) of approved six-cycles from Theorem \ref{thm:5C6} to a bigger constant $k>5$ or, even better, to a function that depends on $d$.
If we observe this from the perspective of Petersen graph $P$, then is seems that the best linear function with respect to $d$ is $d-2$. Namely, $\chi_b(P)=3<d+1$ as shown in \cite{JaKl} and every vertex $x$ from Petersen graph is contained in $2=d-1$ six-cycles from $N_2[x]=V(P)$. So, the next question seems natural.

\begin{question}
How far can the bound on the number of $C_6$ including one particular vertex $x$ from Theorem \emph{\ref{thm:5C6}} be increased? Is the Conjecture \ref{conj} true if we allow $d-2$ six-cycles in $N_2[x]$ that contains $x$?
\end{question}

In the second approach it seems that we did avoid the problems with respect to $P$ due to $d\geq 7$. For $P$ all the bunches (there are three in $P$) have all their neighbors in $N_2[x]=V(P)$. However the approach from Theorem \ref{thm:d:2bunches} does not work as there are not enough rows (and with this also columns) present in $P$. With this it seems that we are on the safe side. However, one needs to find an improvement in the proof to get from two bunches to maybe just one bunch or even better to any two vertices in one bunch that have all the neighbors in $N_2(x)$. The reason for the last suggestion is that in the proof of Theorem \ref{thm:d:2bunches} at the end we have only two b-vertices that do not belong to $N(x)$ and they belong to a bunch with all the neighbors in $N_2(x)$. 

\begin{question}
Is $\chi_b(G)=d+1$ for a $d$-regular graph with $g(G)=5$ where there exists $x\in V(G)$ with only one bunch such that all the neighbors of that bunch are in $N_2(x)$. Even more, are just two vertices from one bunch with the property that its neighbors belong to $N_2(x)$ enough to have $\chi_b(G)=d+1$?
\end{question}

\noindent\textbf{Acknowledgments}: I.P. was partially supported by Slovenian Research and Innovation Agency by research program number P1-0297.


\end{document}